\documentclass[a4paper,11pt]{article}
\pdfoutput=1

\usepackage{graphicx, mathtools, amssymb, latexsym, amsmath, amsfonts, amsthm, bbm, tikz}
\usepackage[toc,page]{appendix}
\usepackage[ruled, linesnumbered, noend]{algorithm2e}
\usepackage{enumitem}

\usepackage[T1]{fontenc}

\usepackage{array, fancyhdr, bm, listings, soul, xcolor,titlesec, cancel}
\usepackage[skip=0.5\baselineskip]{caption}

\usepackage{thmtools, thm-restate}
\usepackage{thm-autoref}

\definecolor{Darkblue}{rgb}{0,0,0.4}
\definecolor{Brown}{cmyk}{0,0.61,1.,0.60}
\definecolor{Purple}{cmyk}{0.45,0.86,0,0}
\definecolor{Darkgreen}{rgb}{0.133,0.543,0.133}
\usepackage[colorlinks,linkcolor=Darkblue,filecolor=blue,citecolor=blue,urlcolor=Darkblue,pagebackref]{hyperref}
\usepackage[nameinlink]{cleveref}

\newcommand{\OO}{\mathcal{O}}

\newcommand{\flow}{\texttt{flow}}
\newcommand{\bag}{\texttt{bag}}
\newcommand{\torso}{\texttt{torso}}
\newcommand{\adh}{\texttt{adh}}

\theoremstyle{definition}
\newtheorem{theorem}{Theorem}

\newtheorem{lemma}[theorem]{Lemma}

\newtheorem{claim}[theorem]{Claim}

\title{On Induced Versions of Menger's Theorem on Sparse Graphs}
\author{
Peter Gartland\thanks{University of California, Santa Barbara, USA. \texttt{petergartland@ucsb.edu}.}
\and
Tuukka Korhonen\thanks{Department of Informatics, University of Bergen, Norway. \texttt{tuukka.korhonen@uib.no}. This work was supported by the Research Council of Norway via the project BWCA (grant no. 314528).}
\and
Daniel Lokshtanov\thanks{University of California, Santa Barbara, USA. \texttt{ daniello@ucsb.edu}.}
}
\date{}

\begin{document}
\maketitle

\begin{abstract}
Let $A$ and $B$ be sets of vertices in a graph $G$.
Menger's theorem states that for every positive integer $k$, either there exists a collection of $k$ vertex-disjoint paths between $A$ and $B$, or $A$ can be separated from $B$ by a set of at most $k-1$ vertices.
Let $\Delta$ be the maximum degree of $G$.
We show that there exists a function $f(\Delta) = (\Delta+1)^{\Delta^2+1}$, so that for every positive integer $k$, either there exists a collection of $k$ vertex-disjoint and pairwise anticomplete paths between $A$ and $B$, or $A$ can be separated from $B$ by a set of at most $k \cdot f(\Delta)$ vertices.
We also show that the result can be generalized from bounded-degree graphs to graphs excluding a topological minor.
On the negative side, we show that no such relation holds on graphs that have degeneracy 2 and arbitrarily large girth, even when $k = 2$. 
Similar results were obtained independently and concurrently by Hendrey, Norin, Steiner, and Turcotte~[arXiv:2309.07905].
\end{abstract}

\section{Introduction}
The classic theorem of Menger~\cite{menger1927allgemeinen} gives a tight duality between separators and collections of vertex-disjoint paths in graphs: The only reason one cannot route a collection of $k$ vertex-disjoint paths between two sets of vertices $A$ and $B$ in a graph is that $A$ can be separated from $B$ by a set of less than $k$ vertices.
The study of induced substructures of graphs (i.e., induced subgraphs, induced subdivisions, and induced minors) motivates the question of an induced version of Menger's theorem: Does the inability to route a collection of $k$ vertex-disjoint and pairwise anticomplete\footnote{We say that two paths $P_i$ and $P_j$ are \emph{anticomplete} if there are no edges between the vertices of $P_i$ and the vertices of $P_j$.} paths between sets of vertices $A$ and $B$ imply the existence of some kind of an $(A,B)$-separator?

Because the problem of finding two vertex-disjoint and anticomplete paths between given sets $A$ and $B$ is NP-complete~\cite{DBLP:journals/dm/Bienstock91} (even on graphs of degree at most $3$~\cite{DBLP:journals/endm/LevequeLMT07}), a tight characterization similar to Menger's theorem appears unlikely.
Despite this, it would be interesting to find approximate reasons for the inability to route such paths.
Recently, Albrechtsen, Huynh, Jacobs, Knappe, and Wollan~\cite{DBLP:journals/corr/abs-2305-04721}, and independently, Georgakopoulos and Papasoglu~\cite{georgakopoulos2023graph} showed that if one cannot route two vertex-disjoint and anticomplete paths between $A$ and $B$, then $A$ can be separated from $B$ by a separator of bounded radius.
They also conjectured that this could be generalized to $k$ paths and separators that are a union of less than $k$ sets of bounded radius.

In this paper, we take a different approach to an induced version of Menger's theorem.
We investigate the question on which graph classes does the inability to route $k$ vertex-disjoint pairwise anticomplete paths between $A$ and $B$ imply the existence of an $(A,B)$-separator of size at most $f(k)$, for some function $f$.
Results similar to ours were obtained independently and concurrently by Hendrey, Norin, Steiner, and Turcotte~\cite{DBLP:journals/corr/abs-2309-07905}.

Let us then state our results.
We start with graphs of bounded degree.

\begin{theorem}
\label{the:degree}
Let $G$ be a graph of maximum degree $\Delta$ and $A,B \subseteq V(G)$.
There exists a function $f(\Delta) = (\Delta+1)^{\Delta^2+1}$ so that for every positive integer $k$, there either exists a collection of $k$ vertex-disjoint and pairwise anticomplete paths between $A$ and $B$, or $A$ can be separated from $B$ by a set of at most $k \cdot f(\Delta)$ vertices.
\end{theorem}

The proof of \Cref{the:degree} is similar to and inspired by a proof of a similar result on treewidth and grid induced minors on bounded-degree graphs~\cite{DBLP:journals/jctb/Korhonen23}.
It remains an interesting open question on how much the upper bound on $f(\Delta)$ could be improved.
We are not aware of any non-trivial lower bounds on it, in particular, to the best of our knowledge it could hold that $f(\Delta) = \OO(\Delta)$.

A graph $G$ contains a graph $H$ as a \emph{minor} if $H$ can be obtained from $G$ by vertex deletions, edge deletions, and edge contractions.
We show that a result similar to \Cref{the:degree} also holds for graphs excluding a minor.

\begin{theorem}
\label{the:minor}
Let $G$ be a graph that excludes a graph $H$ as a minor and $A,B \subseteq V(G)$.
There exists a function $g(H) = \OO(|V(H)| \cdot \sqrt{\log |V(H)|})$ so that for every positive integer $k$, there either exists a collection of $k$ vertex-disjoint and pairwise anticomplete paths between $A$ and $B$, or $A$ can be separated from $B$ by a set of at most $k \cdot g(H)$ vertices.
\end{theorem}

The proof of \Cref{the:minor} follows quite easily from the fact that $H$-minor-free graphs have average degree $\OO(|V(H)| \cdot \sqrt{\log |V(H)|})$~\cite{DBLP:journals/combinatorica/Kostochka84,thomason_1984}.

A graph $G$ contains a graph $H$ as a \emph{topological minor} if $H$ can be obtained from $G$ by vertex deletions, edge deletions, and contractions of edges incident to degree-2 vertices.
Note that topological-minor-free graphs generalize both bounded-degree graphs and minor-free graphs, and therefore the following theorem generalizes both \Cref{the:degree,the:minor}.

\begin{theorem}
\label{the:topminor}
Let $G$ be a graph that excludes a graph $H$ as a topological minor and $A,B \subseteq V(G)$.
There exists a function $h(H)$ so that for every positive integer $k$, there either exists a collection of $k$ vertex-disjoint and pairwise anticomplete paths between $A$ and $B$, or $A$ can be separated from $B$ by a set of at most $k \cdot h(H)$ vertices.
\end{theorem}

The proof of \Cref{the:topminor} is based on combining \Cref{the:degree,the:minor} by using the structure theorem for graphs excluding a topological minor~\cite{DBLP:journals/siamdm/ErdeW19,DBLP:journals/siamcomp/GroheM15}.

We remark that the proofs of all of our results can be turned into polynomial-time algorithms that return either of the outcomes.
For \Cref{the:degree} and \Cref{the:minor} the algorithms are simple, while \Cref{the:topminor} requires the algorithmic version of the structure theorem~\cite{DBLP:journals/siamcomp/GroheM15}.

It can be asked whether such a relation between separators and vertex-disjoint pairwise anticomplete paths would hold in even more general classes than graphs excluding topological minors.
For this, the natural candidate would be the graphs of \emph{degeneracy} at most $d$ for some constant $d$, i.e., the graphs that can be reduced to empty by successively deleting vertices of degree at most $d$.
However, we give a counterexample that leaves not much room for further generalization of such results.
The \emph{girth} of a graph is the length of its shortest cycle.

\begin{theorem}
\label{the:counterexample}
For all integers $k$ and $g$, there exists a graph $G_{k,g}$ of degeneracy $2$ and girth at least $g$, and sets $A,B \subseteq V(G_{k,g})$, so that $A$ cannot be separated from $B$ by fewer than $k$ vertices, but there does not exist two vertex-disjoint anticomplete paths between $A$ and $B$.
\end{theorem}

\section{Preliminaries}
For a positive integer $n$, we denote by $[n]$ the set of integers $\{1,\ldots,n\}$.
We denote the set of vertices of a graph $G$ by $V(G)$ and the set of edges by $E(G)$.
The set of neighbors of a vertex $v$ is denoted by $N(v)$, and the closed neighborhood by $N[v] = N(v) \cup \{v\}$.
For a set of vertices $X \subseteq V(G)$, we denote its neighborhood by $N(X) = \bigcup_{v \in X} N(v) \setminus X$ and the closed neighborhood by $N[X] = N(X) \cup X$.
A graph $G'$ is a subgraph of $G$ if $V(G') \subseteq V(G)$ and $E(G') \subseteq E(G)$.
The subgraph induced by a set of vertices $X \subseteq V(G)$ is denoted by $G[X]$, and we also use $G \setminus X = G[V(G) \setminus X]$.
We use the convention that a connected component of a graph $G$ is a maximal set of vertices $C \subseteq V(G)$ so that $G[C]$ is connected.
We say that two sets of vertices $X,Y \subseteq V(G)$ are anticomplete if there are no edges with one endpoint in $X$ and another in $Y$.

A path in a graph $G$ is a non-empty sequence $P = v_1, \ldots, v_\ell$ of distinct vertices so that consecutive vertices are adjacent.
For sets $A,B \subseteq V(G)$, we say that $P$ is an $(A,B)$-path if $v_1 \in A$ and $v_\ell \in B$.
We denote by $\flow_G(A,B)$ the maximum cardinality of a collection of vertex-disjoint $(A,B)$-paths in $G$.
We say that a set $S \subseteq V(G)$ is an $(A,B)$-separator if every $(A,B)$-path intersects the set $S$.

\begin{theorem}[Menger~\cite{menger1927allgemeinen}]
Let $G$ be a graph and $A,B \subseteq V(G)$.
It holds that $\flow_G(A,B)$ is equal to the minimum size of an $(A,B)$-separator.
\end{theorem}

The contraction of an edge $ab \in E(G)$ means creating a vertex $c$ so that $N(c) = N(a) \cup N(b) \setminus \{a,b\}$, and deleting the vertices $a$ and $b$.
A graph $H$ is a minor of $G$ if $H$ can be obtained from a subgraph of $G$ by edge contractions.
A graph $H$ is a topological minor of $G$ if $H$ can be obtained from a subgraph of $G$ by contractions of edges incident to vertices of degree $2$.

\section{Graphs of bounded degree}
In this section we prove \Cref{the:degree}, i.e., the result for graphs of degree at most $\Delta$.

A distance-3 independent set in a graph $G$ is a set $I \subseteq V(G)$ so that if $u,v \in I$ are distinct vertices in $I$, then the closed neighborhoods $N[u]$ and $N[v]$ are disjoint (equivalently, the distance between $u$ and $v$ is at least $3$).

We first show that the vertices of any given distance-3 independent set can be made to have degree at most $2$ without decreasing $\flow_G(A,B)$ much, and then use the procedure on all vertices by observing that a bounded-degree graph can be partitioned into a bounded number of distance-3 independent sets.

\begin{lemma}
\label{lem:onespars}
Let $G$ be a graph of maximum degree $\Delta$, $A,B \subseteq V(G)$, and let $I \subseteq V(G) \setminus N[A \cup B]$ be a distance-3 independent set in $G$.
There exists an induced subgraph $G[X]$ of $G$ so that
\begin{enumerate}
\item $A,B \subseteq X$,
\item $\flow_{G[X]}(A,B) \ge \flow_G(A,B)/(\Delta+1)$, and
\item every vertex in $I \cap X$ has degree at most $2$ in $G[X]$.
\end{enumerate}
\end{lemma}
\begin{proof}
Because $I$ is a distance-3 independent set, the sets $N[u]$ and $N[v]$ are disjoint for distinct $u,v \in I$.
Moreover, because $I \subseteq V(G) \setminus N[A \cup B]$, they do not contain any vertices from $A \cup B$.
Let $G'$ be the graph obtained by contracting each set $N[v]$ for $v \in I$ into a single vertex.

We have that $\flow_{G'}(A,B) \ge \flow_G(A,B)/(\Delta+1)$, because any $(A,B)$-separator of size $k$ in $G'$ can be mapped into an $(A,B)$-separator of size $k \cdot (\Delta+1)$ in $G$ by expanding it according to the contractions.
Let $t = \flow_{G'}(A,B)$, and let $P'_1, \ldots, P'_t$ be $t$ vertex-disjoint $(A,B)$-paths in $G'$.
By expanding these paths according to the contractions, they can be mapped into $t$ vertex-disjoint $(A,B)$-paths $P_1, \ldots, P_t$ in $G$ with the additional property that for every $v \in I$, at most one such path intersects $N[v]$, and moreover, each path intersects $N[v]$ on at most three vertices.

We then construct the induced subgraph $G[X]$ by deleting some vertices in $N[v]$ for every $v \in I$, without deleting any vertices in the paths $P_1, \ldots, P_t$.
If none of the paths contains $v \in I$, then we delete only $v$ from $N[v]$.
Otherwise, let $P_i$ be a path that contains $v \in I$.
We have that no other path contains vertices in $N[v]$.
Let $\{u,w\} \subseteq N(v)$ be the other two vertices of $P_i$ in $N[v]$.
We delete from $N[v]$ all other vertices except $v$,$u$, and $w$.

It can be verified that we do not delete any vertices in the paths $P_1,\ldots,P_t$, so 
\[\flow_{G[X]}(A,B) \ge t \ge \flow_G(A,B)/(\Delta+1).\]
Also, for every $v \in I$, we either delete it, or delete all but two neighbors of it, so every vertex in $I \cap X$ has degree at most $2$ in $G[X]$.
\end{proof}

We then apply \Cref{lem:onespars} iteratively to make all vertices in $V(G) \setminus N[A \cup B]$ to have degree at most $2$ without decreasing $\flow_G(A,B)$ much.

\begin{lemma}
\label{lem:mainspars}
Let $G$ be a graph of maximum degree $\Delta$, and $A,B \subseteq V(G)$.
There exists an induced subgraph $G[X]$ of $G$ so that
\begin{enumerate}
\item $A,B \subseteq X$,
\item $\flow_{G[X]}(A,B) \ge \flow_G(A,B)/(\Delta+1)^{\Delta^2+1}$, and
\item every vertex in $X \setminus N[A \cup B]$ has degree at most $2$ in $G[X]$.
\end{enumerate}
\end{lemma}
\begin{proof}
Let $Y = V(G) \setminus N[A \cup B]$.
Because the square of $G$ has maximum degree $\Delta^2$, the vertices in $Y$ can be partitioned into $\Delta^2+1$ distance-3 independent sets $I_1, \ldots, I_{\Delta^2+1}$.
Then, we apply \Cref{lem:onespars} one by one to the sets $I_1, \ldots, I_{\Delta^2+1}$.
In particular, let $G_0 = G$.
For each $i \in [\Delta^2+1]$, we obtain $G_i$ from $G_{i-1}$ by applying \Cref{lem:onespars} to the graph $G_{i-1}$ and the distance-3 independent set $I_i \cap V(G_{i-1})$.
Note that taking induced subgraphs does not decrease distances or increase the maximum degree, so all of the invariants required by \Cref{lem:onespars} remain valid.

Let $G[X] = G_{\Delta^2+1}$.
By the properties guaranteed by \Cref{lem:onespars} we have that $A,B \subseteq X$, every vertex in $Y \cap X$ has degree at most $2$ in $G[X]$, and 
\[\flow_{G[X]}(A,B) \ge \flow_G(A,B)/(\Delta+1)^{\Delta^2+1}.\]
\end{proof}

We then finish the proof of \Cref{the:degree}.

\begin{proof}[Proof of \Cref{the:degree}]
First, let us observe that without loss of generality, we can assume that every vertex in $A \cup B$ has degree at most $1$, and every vertex in $N(A) \cup N(B)$ has degree at most $2$.
In particular, while there exists a vertex $a \in A$, so that $a$ has degree more than $1$ or $N(a)$ contains a vertex of degree more than $2$, we can add to $G$ a two-vertex path whose one endpoint is adjacent to $a$, and then move $a$ to the other (degree-1) endpoint.
The same process is applied to the set $B$.
Any $(A,B)$-path in the new graph can be mapped into an $(A,B)$-path in the old graph by simply removing the two first and the two last vertices, and this mapping preserves vertex-disjointness and anticompleteness between two paths.

Then, by Menger's theorem it suffices to show that if $\flow_G(A,B) \ge k \cdot (\Delta+1)^{\Delta^2+1}$, then there exists $k$ vertex-disjoint and pairwise anticomplete paths between $A$ and $B$.
We apply \Cref{lem:mainspars} to the graph $G$, and obtain an induced subgraph $G[X]$ that contains $A \cup B$, has maximum degree $2$, and in which
\[\flow_{G[X]}(A,B) \ge \flow_G(A,B)/(\Delta+1)^{\Delta^2+1} \ge k.\]
Because each vertex in $A \cup B$ has degree at most $1$ and all other vertices have degree at most $2$, any vertex-disjoint $(A,B)$-paths in $G[X]$ are also pairwise anticomplete.
\end{proof}

\section{Graphs excluding a minor}
In this section we prove \Cref{the:minor}, i.e., the result for $H$-minor-free graphs.
The only tool we need is the following fact that minor-free graphs are sparse.

\begin{theorem}[Kostochka~\cite{DBLP:journals/combinatorica/Kostochka84}, Thomason~\cite{thomason_1984}]
\label{thm:minorfreeavgdegree}
If a graph $G$ does not contain a graph $H$ as a minor, then $G$ has average degree at most $\OO(|V(H)| \cdot \sqrt{\log |V(H)|})$.
\end{theorem}

Then we are ready to prove \Cref{the:minor}.

\begin{proof}[Proof of \Cref{the:minor}]
Let $t$ be the size of a smallest $(A,B)$-separator.
If $t \le \OO(k \cdot |V(H)| \cdot \sqrt{\log |V(H)|})$ we are done, so assume $t \ge c \cdot k \cdot |V(H)| \cdot \sqrt{\log |V(H)|}$ for some positive universal constant $c$ to be chosen later.
Now, let $P_1, \ldots, P_t$ be a collection of $t$ vertex-disjoint $(A,B)$-paths that exist by Menger's theorem.
Let $G_P$ be the graph with $t$ vertices $v_1, \ldots, v_t$ so that $v_i$ and $v_j$ for $i \neq j$ are adjacent if there is at least one edge between $P_i$ and $P_j$.

We observe that $G_P$ is a minor of $G$, obtained by contracting each of the paths $P_i$ to one vertex and deleting other vertices.
Therefore, $G_P$ excludes $H$ as a minor, and therefore by \Cref{thm:minorfreeavgdegree}, $G_P$ has average degree $\OO(|V(H)| \cdot \sqrt{\log |V(H)|})$, implying that $G_P$ contains an independent set of size at least
\begin{align*}
\frac{t}{\OO(|V(H)| \cdot \sqrt{\log |V(H)|})} &\ge \frac{c \cdot k \cdot |V(H)| \cdot \sqrt{\log |V(H)|}}{\OO(|V(H)| \cdot \sqrt{\log |V(H)|})},
\end{align*}
which is at least $k$ after choosing $c$ to be large enough.
An independent set of $G_P$ of size $k$ corresponds to a collection of $k$ disjoint pairwise anticomplete $(A,B)$-paths.
\end{proof}

\section{Graphs excluding a topological minor}
In this section we prove \Cref{the:topminor}, i.e., the result for $H$-topological-minor-free graphs.
We use the structure theorem for $H$-topological-minor-free graphs, for which we need to recall some definitions.

A \emph{tree decomposition} of a graph $G$ is a pair $(T,\bag)$, where $T$ is a tree and $\bag \colon V(T) \rightarrow 2^{V(G)}$ is a function satisfying
\begin{enumerate}
\item for every edge $uv \in E(G)$, there exists a node $x \in V(T)$ so that $\{u,v\} \subseteq \bag(x)$ and
\item for every vertex $v \in V(G)$, the set of nodes $\{x \mid v \in \bag(x)\}$ induces a non-empty connected subtree of $T$.
\end{enumerate}

Let $(T,\bag)$ be a tree decomposition of $G$ and $xy \in E(T)$ be an edge of $T$.
The \emph{adhesion} at $xy$ is defined as $\adh(xy) = \bag(x) \cap \bag(y)$.
Let $x \in V(T)$.
The \emph{torso} at $x$ is defined as the graph obtained from $G[\bag(x)]$ by adding edges to make every adhesion of an edge incident to $x$ into a clique, and is denoted by $\torso(x)$.

We are ready state the structure theorem.

\begin{theorem}[Grohe and Marx~\cite{DBLP:journals/siamcomp/GroheM15}, see also Erde and Wei{\ss}auer~\cite{DBLP:journals/siamdm/ErdeW19}]
\label{thm:topminorstructure}
Let $G$ be a graph that excludes a graph $H$ as a topological minor.
There exists a function $f(H)$ so that $G$ admits a tree decomposition $(T,\bag)$, where 
\begin{enumerate}
\item every adhesion has size at most $f(H)$ and
\item for every node $x \in V(T)$, $\torso(x)$ either excludes $K_{f(H)}$ as a minor or has at most $f(H)$ vertices of degree at least $f(H)$.
\end{enumerate}
\end{theorem}

We then prove \Cref{the:topminor}.
We first consider a case where the decomposition $(T,\bag)$ given by \Cref{thm:topminorstructure} contains a single ``central'' node $c \in V(T)$ so that $\bag(c)$ is an $(A,B)$-separator.
Afterwards, the general case will be reduced to this.

\begin{lemma}
\label{lem:topmin}
Let $G$ be a graph and $A,B \subseteq V(G)$, and let $(T,\bag)$ be a tree decomposition of $G$ and $r$ an integer so that
\begin{enumerate}
\item every adhesion of $(T,\bag)$ has size at most $r$,
\item\label{lem:topmin:struc} for every node $x \in V(T)$, $\torso(x)$ either excludes $K_r$ as a minor or has at most $r$ vertices of degree at least $r$, and
\item\label{lem:topmin:cnode} there exists a node $c \in V(T)$ so that $\bag(c)$ is an $(A,B)$-separator.
\end{enumerate}
Then there exists a function $g(r)$ so that there exists a collection of at least $\flow_G(A,B)/g(r)$ vertex-disjoint and pairwise anticomplete paths between $A$ and $B$.
\end{lemma}
\begin{proof}
Let $c \in V(T)$ be the special central node guaranteed by \Cref{lem:topmin:cnode}.
We construct a graph $G'$ and sets $A'$ and $B'$ as follows.
The graph $G'$ is constructed from $G[\bag(c)]$ by making for each connected component $C$ of $G \setminus \bag(c)$ the neighborhood $N(C)$ into a clique.
Observe that $G'$ is a subgraph of $\torso(c)$ because each such $N(C)$ is contained in $\adh(cx)$ for a neighbor $x$ of $c$.

Then, to construct the sets $A'$ and $B'$, we iterate over all connected components $C$ of $G \setminus \bag(c)$.
If $C$ contains a vertex from $A$, then we add all vertices in $N(C)$ to $A'$, and if $C$ contains a vertex from $B$, then we add all vertices in $N(C)$ to $B'$.
Finally, we add to $A'$ all vertices in $\bag(c) \cap A$ and to $B'$ all vertices in $\bag(c) \cap B$.

We then show that $(A,B)$-paths in $G$ can be mapped into $(A',B')$-paths in $G'$.

\begin{claim}
\label{lem:topmin:claimflow}
It holds that $\flow_{G'}(A',B') \ge \flow_G(A,B)$.
\end{claim}
\begin{proof}[Proof of the claim.]
Let $t = \flow_G(A,B)$ and let $P_1, \ldots, P_t$ be a collection of $t$ vertex-disjoint $(A,B)$-paths in $G$.
We claim that the intersection $P_i \cap \bag(c)$ obtained simply by removing from $P_i$ all vertices outside of $\bag(c)$ is an $(A',B')$-path in $G'$.
First, $P_i \cap \bag(c)$ is non-empty because $\bag(c)$ is an $(A,B)$-separator in $G$.
Then, we observe that parts of $P_i$ going outside of $\bag(c)$ can be shortcutted because we made the neighborhoods of the components of $G \setminus \bag(c)$ into cliques.
Finally, the first vertex in $P_i \cap \bag(c)$ is in $A'$ because if it was outside of $\bag(c)$, then we added the neighborhood of the component to $A'$, and by a similar argument the last vertex in $P_i \cap \bag(c)$ is in $B'$.
\end{proof}

Then, we argue that there is a function $g(r)$ so that there exists a collection of at least $\flow_{G'}(A',B')/g(r)$ vertex-disjoint and pairwise anticomplete $(A',B')$-paths in $G'$.
We consider two cases: Either $G'$ excludes $K_r$ as a minor or $G'$ has at most $r$ vertices of degree at least $r$.
One of these cases holds by \Cref{lem:topmin:struc}.
If $G'$ excludes $K_r$ as a minor, then $g(r)$ exists by \Cref{the:minor}.
If $G'$ has at most $r$ vertices of degree at least $r$, then $g(r)$ exists by first deleting the vertices of degree at least $r$, and then applying \Cref{the:degree}.

It remains to argue that a collection of vertex-disjoint pairwise anticomplete $(A',B')$-paths in $G'$ can be turned into a collection of vertex-disjoint pairwise anticomplete $(A,B)$-paths in $G$.
For this, the main observation is that because the paths are pairwise anticomplete, for each component $C$ of $G \setminus \bag(c)$, at most one of the paths intersects $N(C)$.
This allows to route the parts of the paths using edges in $N(C)$ that do not exists in $G$ via the components $C$.
In particular, as at most one path intersects $N(C)$, at most one path will be routed via the component $C$.
Also, if a path starts in a vertex $a \in A'$ that was added to $A'$ because it was in $N(C)$ for some $C$ containing a vertex in $A$, we can route the path from $a$ to $C \cap A$ arbitrarily in $C$.
Similarly, we route the ends of paths to vertices in $B$.
\end{proof}

We then finish the proof of \Cref{the:topminor} by ``chopping'' the graph by adhesions until every connected component satisfies the condition of \Cref{lem:topmin}.

\begin{proof}[Proof of \Cref{the:topminor}.]
Suppose $G$ excludes $H$ as a topological minor.
Our goal is to show that there exists a function $h(H)$ so that for every positive integer $k$, there either exists a collection of $k$ vertex-disjoint and pairwise anticomplete paths between $A$ and $B$, or $A$ can be separated from $B$ by at most $k \cdot h(H)$ vertices.
Let $f(H)$ be the function and $(T,\bag)$ the tree decomposition of $G$ given by the structure theorem (\Cref{thm:topminorstructure}).
Let $g(r)$ be the function given by \Cref{lem:topmin}.
We will show that we can set $h(H) = g(f(H)) + f(H)$.
If $\flow_G(A,B) \le k \cdot h(H)$ we are done, so suppose that $\flow_G(A,B) \ge k \cdot h(H) \ge k \cdot g(f(H)) + k \cdot f(H)$.

We do the following process to delete vertices from $G$.
As long as $G$ has a connected component $C$ so that $C$ contains vertices from both $A$ and $B$ and there exists an adhesion $\adh(xy)$ of $(T,\bag)$ so that $G[C \setminus \adh(xy)]$ has at least two connected components that contain vertices from both $A$ and $B$, we delete the set $\adh(xy) \cap C$ from $G$.

We observe that at every such deletion, the number of connected components of $G$ that contain vertices from both $A$ and $B$ increases by at least one.
Therefore, if there are more than $k$ such deletions we are ready, because then we can route each of the paths in a different component, so assume there are at most $k$ such deletions.
Let $G'$ be the graph obtained at the end of the process and $A' = V(G') \cap A$ and $B' = V(G') \cap B$.
Because there are at most $k$ such deletions and the adhesions have size at most $f(H)$, we have that $\flow_{G'}(A',B') \ge \flow_G(A,B)-k \cdot f(H) \ge k \cdot g(f(H))$.

Then we argue that we can apply \Cref{lem:topmin} to every connected component of $G'$.
\begin{claim}
\label{claim:chopping}
Let $C$ be a connected component of $G'$.
There exists a node $c \in V(T)$ so that $\bag(c) \cap C$ is an $(A' \cap C, B' \cap C)$-separator in $G'$.
\end{claim}
\begin{proof}[Proof of the claim.]
Let $xy$ be an edge of the tree $T$.
Let $T_x \subseteq V(T)$ be the set of nodes of $T$ that are closer to $x$ than $y$ (including the node $x$), and $T_y \subseteq V(T)$ the nodes that are closer to $y$ than $x$(including the node $y$).
Let $C_x = \bigcup_{z \in T_x} \bag(z) \cap C$ and $C_y = \bigcup_{z \in T_y} \bag(z) \cap C$.
We have that $\adh(xy) \cap C$ is a $(C_x,C_y)$-separator in $G'$.
Therefore, by our deletion process, at most one of $G'[C_x \setminus \adh(xy)]$ and $G'[C_y \setminus \adh(xy)]$ contains an $(A',B')$-path.

If neither of them contains an $(A',B')$-path, then $\adh(xy) \cap C$ is an $(A' \cap C, B' \cap C)$-separator in $G'$, and therefore $\bag(x) \cap C \supseteq \adh(xy) \cap C$ is an $(A' \cap C, B' \cap C)$-separator in $G'$ and we are done.
Otherwise, if $G[C_x \setminus \adh(xy)]$ contains an $(A',B')$-path, we orient the edge $xy$ towards $x$, and if $G[C_y \setminus \adh(xy)]$ contains an $(A',B')$-path, we orient the edge $xy$ towards $y$.

Because $T$ is a tree, by a walking argument there exists a node $c \in V(T)$ so that all edges incident to $c$ are oriented towards $c$ in $T$.
We claim that $\bag(c) \cap C$ is an $(A' \cap C, B' \cap C)$-separator in $G'$.
If there would be an $(A' \cap C, B' \cap C)$-path in $G'[C \setminus \bag(c)]$, then it would be contained in the graph $G'[C_y \setminus \adh(cy)]$ for some $cy \in E(T)$.
However, this would contradict that $cy$ is oriented towards $c$.
Therefore, every $(A' \cap C, B' \cap C)$-path in $G'$ intersects $\bag(c)$.
\end{proof}

Let $C$ be a connected component of $G'$.
We note that the restriction $(T, \bag\restriction_C)$ is a tree decomposition of $G'[C]$ that still satisfies the properties guaranteed by \Cref{thm:topminorstructure}, and also the node $c \in V(T)$ given by \Cref{claim:chopping} still gives an $(A' \cap C, B' \cap C)$-separator in $G'[C]$.
Therefore, we can apply \Cref{lem:topmin} with $G'[C]$, $(T, \bag\restriction_C)$, and the node $c$, and obtain a collection of at least $\flow_{G'[C]}(A' \cap C, B' \cap C)/g(f(H))$ vertex-disjoint pairwise anticomplete $(A' \cap C, B' \cap C)$-paths.
As the sum of $\flow_{G'[C]}(A' \cap C, B' \cap C)/g(f(H))$ over all components $C$ is $\flow_{G'}(A', B')/g(f(H)) \ge k$, we are done.
\end{proof}

\section{Graphs of bounded degeneracy}
In this section we prove \Cref{the:counterexample}, i.e., that for all integers $k$ and $g$, there exists a graph $G_{k,g}$ of degeneracy $2$ and girth at least $g$, and sets $A,B \subseteq V(G_{k,g})$, so that $A$ cannot be separated from $B$ by fewer than $k$ vertices, but there does not exist two vertex-disjoint anticomplete paths between $A$ and $B$.

\begin{figure}[htb]
\begin{center}
\includegraphics[width=0.65\textwidth]{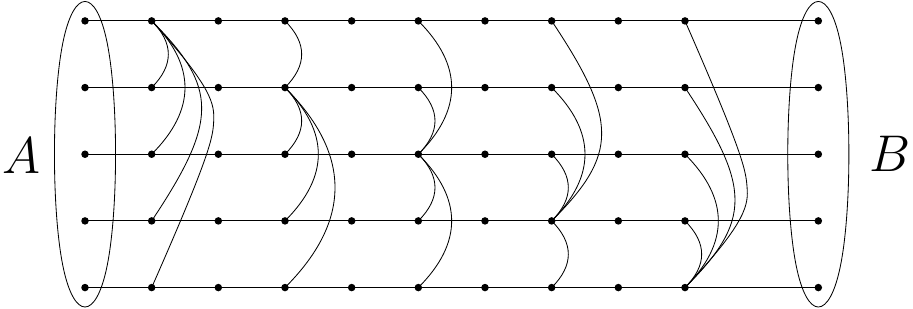}
\caption{The construction of the graph in the proof of \Cref{the:counterexample} for $k=5$ and $p=2$.}
\label{fig:example}
\end{center}
\end{figure}

\begin{proof}[Proof of \Cref{the:counterexample}.]
Let $p = \max(2, \lfloor g/2 \rfloor)$.
We next describe the construction of the graph $G_{k,g}$, see \Cref{fig:example} for an illustration of it.
The vertices of $G_{k,g}$ are $\{v_{i,j} \mid i \in [k], j \in [pk+1]\}$.
For the edges, first we add edges between $v_{i,j}$ and $v_{i,j+1}$ for every $i \in [k]$ and $j \in [pk]$, making the sequence $v_{i,1},\ldots,v_{i,pk+1}$ a path for each $i \in [k]$.
Then, for each $i \in [k]$, we make the vertex $v_{i,pi}$ adjacent to $v_{j,pi}$ for all $j \neq i$.
The set $A$ is $A = \{v_{i,1} \mid i \in [k]\}$ and the set $B$ is $B = \{v_{i,pk+1} \mid i \in [k]\}$.

The paths $v_{i,1},\ldots,v_{i,pk+1}$ for $i \in [k]$ verify that $A$ cannot be separated from $B$ by fewer than $k$ vertices.
Then, let $P$ be an $(A,B)$-path.
We argue that $N[P]$ is an $(A,B)$-separator, and therefore there is no other $(A,B)$-path that is vertex-disjoint and anticomplete with $P$.
First, we observe that the set $\{v_{i,pi} \mid i \in [k]\}$ is an $(A,B)$-separator, and therefore $P$ must contain at least one vertex $v_{i,pi}$ for some $i$.
However, we also have that for any $i \in [k]$, the set $N[v_{i,pi}]$ is an $(A,B)$-separator, and therefore $N[P]$ is an $(A,B)$-separator.

We can observe that after deleting vertices of degree at most two the graph becomes a forest, and therefore its degeneracy is at most two.
Also, the girth can be observed to be $2p+2$.
\end{proof}

\bibliographystyle{plain}
\bibliography{paper}

\end{document}